\documentclass[10pt]{amsart}%  
\usepackage{amsmath}
\usepackage{amsfonts}
\usepackage{amsthm}
\usepackage{amssymb}
\usepackage{graphicx}%
\newtheorem{theorem}{Theorem}
\newtheorem*{theorem*}{Theorem}
\newtheorem*{acknowledgement*}{Acknowledgement}
\newtheorem*{definition*}{Definition}

\newtheorem{corollary}[theorem]{Corollary}

\newtheorem{lemma}[theorem]{Lemma}
\newtheorem*{notation*}{Notation}

\newtheorem{proposition}[theorem]{Proposition}

\newcommand{\RR}[0]{{\mathbb{R}}}
\newcommand{\pd}[2]{\frac{\partial #1}{\partial#2}}

\newcommand{\pdt}[0]{\frac{\partial}{\partial t}}

\newcommand{\gt}[0]{\tilde{g}}

\newcommand{\Gam}[0]{\Gamma}
\newcommand{\Gamt}[0]{\widetilde{\Gamma}}

\newcommand{\Rc}[0]{\operatorname{Rc}}
\newcommand{\Rm}[0]{\operatorname{Rm}}

\newcommand{\Xt}[0]{\widetilde{X}}
\newcommand{\Et}[0]{\widetilde{E}}
\newcommand{\Vt}[0]{\widetilde{V}}
\newcommand{\Pt}[0]{\widetilde{P}}

\newcommand{\Rmt}[0]{\widetilde{\operatorname{Rm}}}

\newcommand{\dfn}[0]{\doteqdot}

\newcommand{\wt}[1]{\widetilde{#1}}
\newcommand{\Rt}[0]{\wt{R}}
\newcommand{\nabt}[0]{\wt{\nabla}}

\newcommand{\Deltat}[0]{\widetilde{\Delta}}

\newcommand{\Ec}[0]{\mathcal{E}}

\newcommand{\Gc}[0]{\mathcal{G}}
\newcommand{\Hc}[0]{\mathcal{H}}

\newcommand{\Kc}[0]{\mathcal{K}}
\newcommand{\Lc}[0]{\mathcal{L}}

\newcommand{\Wc}[0]{\mathcal{W}}
\newcommand{\Xc}[0]{\mathcal{X}}
\newcommand{\Yc}[0]{\mathcal{Y}}

\newcommand{\St}[0]{\tilde{S}}

\title[Uniqueness for higher-order geometric flows]
{An energy approach to uniqueness for higher-order geometric flows}

\author{Brett Kotschwar}
\address{Arizona State University, Tempe, AZ, USA}
\email{kotschwar@asu.edu}

\date{Aug 2014}

\begin{document}

\begin{abstract}
We describe a simple, direct method to prove the uniqueness of solutions to a broad class of parabolic geometric evolution
equations. Our argument, which is based on a prolongation procedure and the consideration of 
certain natural energy quantities, does not require the solution of any auxiliary parabolic systems. In previous work, we used a variation of this technique to 
give an alternative proof of the uniqueness of complete solutions to the Ricci flow of uniformly bounded curvature. Here
we extend this approach to curvature flows of all orders, including the $L^2$-curvature flow
and a class of quasilinear higher-order flows related to the obstruction tensor. We also detail
its application to the fully nonlinear cross-curvature flow.
\end{abstract}


\maketitle

\section{Introduction}
We revisit the problem of uniqueness for parabolic curvature flows
on a smooth manifold $M$. The invariance of such curvature flows under $\operatorname{Diff}(M)$
implies that the systems of equations they represent are never strictly parabolic. 
This basic geometric degeneracy must be circumvented in some fashion in order to apply standard techniques from parabolic theory.
The usual device for this purpose is the DeTurck trick (\cite{DeTurck} 
\cite{HamiltonSingularities}). As it concerns uniqueness, DeTurck's trick is at its heart an exchange of one nonstandard problem for two standard ones: the problem of
of uniqueness
for the typically weakly-parabolic curvature flow for separate problems of existence and uniqueness for two auxiliary
strictly parabolic systems.
%For example, the problem of uniqueness of solutions to the Ricci flow
%\cite{HamiltonSingularities}, \cite{ChenZhu}), can be exchanged for
%that of the short-time existence of a certain harmonic-map-type parabolic system and the uniqueness
%of (a suitably well-behaved class of) solutions to the Ricci-DeTurck flow.

When $M$ is compact, or the problems for the auxiliary
equations otherwise belong to established theory, this is a purely advantageous exchange. (See \cite{BahuaudHelliwellU}, e.g., 
for a careful treatment of this approach to a class of higher-order equations.) When $M$ is noncompact, however,
the existence and uniqueness of the solutions to the auxiliary equations may themselves be nontrivial questions. For example, the implementation of the DeTurck trick carried out in \cite{ChenZhu} for 
complete noncompact solutions to the Ricci flow
requires the authors to prove a novel and rather delicate existence theorem to obtain solutions, with suitable estimates, 
to a certain harmonic map-type heat flow. (Their
approach has the additional benefit, however, of establishing the existence of a well-controlled 
solution to the Ricci-DeTurck flow associated to every complete solution to the Ricci flow of bounded curvature.)

We demonstrate here that it is possible to prove the uniqueness of solutions
to many curvature flows directly, without the introduction of any auxiliary parabolic system.
The idea is to reframe the problem
as one for a prolonged system to which a classical $L^2$-energy argument can be applied.
As in the strictly-parabolic case, this method moreover yields a coarse quantitative bound on the $L^2$-norm
of the difference of the solutions.  This is an extension of the technique developed in our earlier paper \cite{KotschwarRFU}
for the Ricci flow and a somewhat more general class of second-order systems. Bell \cite{BellConformalRF} has also recently applied this technique
to the conformal Ricci flow. The purpose of this paper is to generalize this technique to higher-order equations.  

For concreteness, we first consider a class of quasilinear higher-order parabolic equations 
considered in the recent papers of
Bahuaud and Helliwell \cite{BahuaudHelliwellSTE, BahuaudHelliwellU} and give an alternative proof
of the uniqueness result in the latter reference. This class includes 
the $L^2$-curvature flow \cite{StreetsL2} and a family of equations introduced in the papers \cite{Bour}, \cite{BahuaudHelliwellSTE}, \cite{BahuaudHelliwellU}
motivated by the study of the ambient obstruction tensor.  We then formulate a general uniqueness result, and consider a
its application to a further example, the fully nonlinear cross-curvature flow \cite{ChowHamiltonXCF}.  

We will restrict our attention here to compact manifolds where the demonstration of our technique is simplest
and most transparent. With suitable weight functions and estimates on the solutions and the addition of sufficiently rapidly decaying 
weights to the integral quantities, however, this technique can be extended to the noncompact setting. We have
explored this extension already in \cite{KotschwarRFU} for the Ricci flow, where we prove a slight generalization
of the standard uniqueness theorem of \cite{Hamilton3D} and \cite{ChenZhu}
to a class of solutions with some spatial growth of curvature and some potential blow-up of curvature as $t\searrow 0$.

\section{A class of quasilinear curvature flows}

In this section, we demonstrate the energy technique on a class of curvature flows considered
by Bahuaud and Helliwell \cite{BahuaudHelliwellSTE}, \cite{BahuaudHelliwellU}.  The equations have the form
\begin{equation}\label{eq:kcf}
      \pdt g = \Theta_{2k}(g) + \Lambda_{2k-1}(g)
\end{equation}
where the leading order terms satisfy
\[
 \Theta_{2k}(g) \dfn (-1)^{k+1}2(\Delta^{(k)}\Rc + \alpha \Delta^{(k)} S g + \beta \Delta^{(k-1)}\nabla\nabla S)
\]
and the lower-order terms $\Lambda_{2k-1}(g)$ are some polynomial expression in $g$, $g^{-1}$ and the covariant
derivatives of $R$ up to order $2k-1$. 
Here, $\alpha$ and $\beta$ are constants, and $\nabla = \nabla_g$, $\Delta = \Delta_g$, $R = \Rm(g)$, $\Rc = \Rc(g)$, and $S = S(g)$ denote, respectively, 
the Levi-Civita connection, Laplacian, $(4, 0)$-curvature tensor,
 Ricci curvature tensor, and scalar curvature associated to $g$.  We permit $k=0$ when $\beta =0$.

Included in this family of equations are the Ricci flow
 \begin{equation}\label{eq:rf}
   \pdt g = -2\Rc(g),
 \end{equation}
and the  $L^2$-curvature flow of J.\ Streets \cite{StreetsL2}, 
 \begin{equation}\label{eq:l2}
   \pdt g_{ij} = 2 (\Delta R_{ij} -\frac{1}{2}\nabla_i\nabla_j S) + 2R^{pq}R_{ipqj} -R^p_iR_{pj} + R_{i}^{pqr}R_{jpqr} -\frac{1}{4}|\Rm|^2g_{ij},
 \end{equation}
which is the negative gradient flow of the squared $L^2$-norm of $R$.

When $n\geq 4$ is even, $k=n/2 -1$, and $\beta = -(n-2)/(2(n-1))$, for an appropriate choice of $\Lambda$, equation \eqref{eq:kcf} is the so-called \emph{obstruction flow}
\[
  \pdt g = c_n\left(\mathcal{O}_n + (-1)^{\frac n 2}\left(\alpha +\frac{1}{2(n-1)}\right)\Delta^{\frac{n}{2} - 1}Sg \right),
\]
where
\[
 \mathcal{O}_n = \frac{1}{(-2)^{\frac{n}{2}}(n/2 - 2)!}\left(\Delta^{\frac{n}{2}-1}P -\frac{1}{2(n-1)}\Delta^{\frac{n}{2}-2}\nabla\nabla S\right)
\]
is the ambient obstruction tensor of Fefferman-Graham \cite{FeffermanGraham}. Here,
  $P = (n-2)^{-1}(\Rc - (1/(2(n-1))Sg)$ is the Schouten tensor. When $n=4$, the tensor $\mathcal{O}_4$ is the Bach tensor  
\[
 B_{ij} = \Delta P_{ij} - \nabla_i\nabla_j P - 2P^{kl}W_{kijl} + |P|^2g_{ij} - 4P^{2}_{ij}.
\]
These flows were introduced by Bour (in the case $k=1$) and by Bahuaud and Helliwell in the case $k \geq 2$
as a part of a proposed strategy to obtain obstruction-flat manifolds dynamically. Short-time existence is not known for the flow 
by the pure obstruction tensor.

The following theorem is due to Bahuaud-Helliwell \cite{BahuaudHelliwellU}. We will give an alternative
proof based on Theorem \ref{thm:energyder} below.
\begin{theorem}[\cite{BahuaudHelliwellU}, Theorem A]\label{thm:hocfuniqueness} 
Let $M= M^n$ be a compact manifold and $\alpha > -1/(2(n-1))$, $\beta \in \RR$, and $\Omega > 0$ given constants.
If $g$ and $\gt$ are solutions to \eqref{eq:kcf} on $M\times [0, \Omega]$
which satisfy that $g(0) = \gt(0)$, then $g \equiv \gt$ on $M\times [0, \Omega]$. 
\end{theorem}

The idea of our argument is to embed the problem of uniqueness for the solutions of \eqref{eq:kcf} into one for the uniqueness
of solutions to a larger system centered on the differences $\nabla^{(l)}R-\nabt^{(l)}\Rt$ of the covariant derivatives of the curvature tensors.
Individually, $\nabla^{(l)}R$ and $\nabt^{(l)}\Rt$ will satisfy evolution equations that are parabolic (or nearly so), 
but with respect to \emph{different} operators depending on  $g$ and $\gt$. The equation for their difference will therefore involve
the derivatives $\nabla^{(l)}(g - \gt)$ of the difference of the solutions up to order $2k$. 
To obtain a closed system, we will thus need to add to the system sufficiently many derivatives of $g- \gt$. 
Although they will not themselves satisfy parabolic equations, the derivatives of $g -\gt$ can be controlled
pointwise, via their evolution equations,
by sufficiently many differences of the derivatives of the curvature tensors. 
With the right mix of these higher- and lower-order components, we obtain a closed system of mixed differential inequalities
to which standard energy methods can be applied. This prolongation procedure has its origins in the work of Alexakis \cite{Alexakis} and Wong-Yu 
\cite{WongYu}.

For example, in the case of the $L^2$-curvature flow \eqref{eq:l2}, our prolonged system will consist
of the quantities $g-\gt$, $\nabla( \Gamma - \Gamt)$,
$R - \Rt$, and $\nabla^{(2)}R - \nabt^{(2)}\Rt$.  The rest of the argument is just a straightforward computation
to show that
\[
    \Ec(t) \dfn \int_{M}\left(|g-\gt|^2_{g(t)} + |\nabla(\Gamma - \Gamt)|^2_{g(t)} + |R - \Rt|^2_{g(t)} + |\nabla^{(2)}R - \nabt^{(2)}\Rt|_{g(t)}^2\right)\,d\mu_{g(t)}  
\]
satisfies the differential inequality $\dot{\Ec} \leq C\Ec$ for some $C$.

\subsection{The energy quantity}
Let $M= M^n$ now be a compact manifold, and $g(t)$, $\gt(t)$ are two smooth solutions to \eqref{eq:kcf} for some fixed $k\geq 1$
 on $M\times [0, \Omega]$. We will denote their Levi-Civita connections and $(4, 0)$-curvature tensors
by $\nabla$ and $\nabt$ and $\Rm$ and $\Rmt$, respectively. For any $l=0, 1, \ldots$, we define the tensors
\[
h \dfn g- \gt, \quad A \dfn \nabla - \nabt,\quad X^{(l)} \dfn \nabla^{(l)}\Rm - \nabt^{(l)}\Rmt,\quad Z^{(l)} \dfn \nabla^{(l)}S - \nabt^{(l)}\St.  
\]
Here $A$ is the family of $(2, 1)$-tensors given in local coordinates by $A_{ij}^k = \Gamma_{ij}^k - \Gamt_{ij}^k$.  

We will
use one of the solutions, $g(t)$, as a reference metric and, for each $t\in [0, \Omega]$, consider the $L^2(d\mu_{g(t)})$-inner product
induced on various tensor bundles over $M$.  Given a bundle $\mathcal{Z}$ over $M$ and smooth sections $U$, $V\in \mathcal{Z}$,
we will write 
\[
   (U, V) \dfn (U, V)_{L^2(d\mu_{g(t)})}\dfn  \int_{M}\langle U, V\rangle_{g(t)}\,d\mu_{g(t)}, \quad \|U\|^2 \dfn (U, U).
\]
Below, we will simply use 
the unadorned notation 
$|\cdot| \dfn |\cdot|_{g(t)}$ and $d\mu = d\mu_{g(t)}$ to represent, respectively, the norms induced by $g(t)$ on various tensor bundles 
and the Riemannian measure of $g(t)$. We will also use $U\ast V$ to denote a linear combination of contractions of $U\otimes V$ by the metric $g(t)$.  (Should we wish to emphasize
the role of the metric in the contraction or denote contractions by the metric $\gt$, we will use instead the notation
$U\ast_gV$ and $U\ast_{\gt}V$.)

Define the quantities
\begin{align*}
  \Gc(t) &\dfn \|h\|^2 + \|\nabla^{(k)}A\|^2, \quad
  \Hc(t) \dfn \|X^{(0)}\|^2 +  \|X^{(2k)}\|^2, \quad
  \Kc(t) \dfn \|Z^{(2k)}\|^2.
\end{align*}
We will show below that these quantities satisfy the following system of inequalities.
\begin{theorem}\label{thm:energyder} 
For any $\epsilon > 0$, there exists a constant $C$ such that the quantities $\Gc$, $\Hc$, and $\Kc$ satisfy the differential inequalities
\begin{align}
 \begin{split}\label{eq:gder}
  \dot{\Gc} &\leq C(\Gc + \Hc) + 2\epsilon\|\nabla^{(k+1)}X^{(2k)}\|^2
 \end{split}\\
 \begin{split}\label{eq:hder}
  \dot{\Hc} &\leq C(\Gc + \Hc) - 2(1 - \epsilon)\|\nabla^{(k+1)}X^{(2k)}\|^2 - 4\alpha\|\nabla^{(k+1)}Z^{(2k)}\|^2
\end{split}\\
 \label{eq:kder}
  \dot{\Kc} &\leq C(\Gc + \Hc) + 2\epsilon\|\nabla^{(k+1)}X^{(2k)}\|^2  -2(1 +2\alpha(n-1))\|\nabla^{(k+1)}Z^{(2k)}\|^2
\end{align}
on the interval $[0, \Omega]$.
\end{theorem}

Theorem \ref{thm:hocfuniqueness} may then be proven by considering an appropriate combination of $\Gc$, $\Hc$, and $\Kc$.
\begin{corollary}\label{cor:eder}
 For any $r \geq 0$ and any $\epsilon > 0$, there is a constant $C$ depending on $\epsilon$
and the solutions $g$ and $\gt$ such that $\Ec \dfn \Gc + \Hc + r \Kc$
satisfies
\begin{equation}\label{eq:eev2}
 \dot{\Ec} \leq C(1+r)\Ec + a(\epsilon, r)\|\nabla^{(k+1)}X^{(2k)}\|^2  + b(\alpha, n, r)\|\nabla^{(k+1)}Z^{(2k)}\|^2
\end{equation}
on $[0, \Omega]$ where 
\[
  a(\epsilon, r) =  -2(1 - \epsilon(r+2)), \quad b(\alpha, n, r) = -2(2\alpha + r(1 + 2\alpha(n-1))).
\]
In particular, when $\alpha + 1/(2(n-1)) > 0$, choosing
   $r > -2\alpha/(1+2\alpha(n-1))$ and  $\epsilon <  1/(r+2)$, we have $a$, $b < 0$,
so $\dot{\Ec}(t) \leq C\Ec(t)$ on $[0, \Omega]$ and
\[
 \Ec(t_2) \leq e^{C(t_2 - t_1)}\Ec(t_1)
\]
for any $0\leq t_1 \leq t_2 \leq \Omega$.
\end{corollary}

 When $\alpha \geq 0$, as, e.g., in the $L^2$-curvature flow, one can take $r = 0$ and drop the term $\Kc$ from the quantity $\Ec$ altogether.
It is included only to 
balance the potentially positive contribution proportional to $\|\nabla^{(k+1)}Z^{(2k)}\|^2$ 
in the evolution equation \eqref{eq:hder} for $\Hc$. A similar device was used by Bour in \cite{Bour} 
to obtain $L^2$-Bernstein-Bando-Shi-type estimates for solutions to equation \eqref{eq:kcf} with $k=1$.

\subsection{Evolution equations and commutation relations}

The verification of the inequalities \eqref{eq:gder} - \eqref{eq:kder} is a long but straightforward computation,
for which we will first need to collect the evolution equations for the connection and for the covariant derivatives
of the curvature tensor of a solution of \eqref{eq:kcf}.

We begin with two standard commutation identities. Here and below,  we will use the notation $P_g^l(R)$ to denote a polynomial expression formed from sums of contractions of the tensor products of
various factors of $R, \nabla R, \ldots, \nabla^{(l)}R$ taken with respect to the metric $g$.

\begin{lemma}\label{lem:commrel} Suppose $g = g(t)$ is a smooth solution to \eqref{eq:kcf} and $W = W(t)$ is a smooth family of tensor fields on $M$.
Then, for any multi-indices $\beta$ and $\gamma$ of lengths $l$ and $m$, respectively, we have the commutation
relations
\begin{equation}\label{eq:commrel}
 \left[(\nabla)_{\beta}, (\nabla)_{\gamma}\right]W
      =\sum_{p=0}^{l+m-2}g^{-1}\ast\nabla^{(l+m-2-p)}R \ast \nabla^{(p)}W,
\end{equation}
and
\begin{equation}\label{eq:dtdm}
\left[\pdt , \nabla^{(r)}\right]W = \sum_{p=1}^{r}P_g^{2k+p+1}(R)\ast \nabla^{(r-p)}W,
\end{equation}
 for all $r \geq 0$. Here $(\nabla)_{\beta}$ is shorthand for
$\nabla_{\beta_1}\nabla_{\beta_2}\cdots\nabla_{\beta_l}$.
\end{lemma}
\begin{proof}
 Trivially, we have $[\nabla_{\beta_1}, \nabla_{\gamma_1}]W = g^{-1}\ast R\ast W$,
 and so the special case of \eqref{eq:commrel} where $m=1$ then follows
by induction on $l$ from the identity
\[
  [\nabla_{\beta_1}\nabla_{\beta_2}\cdots\nabla_{\beta_{l+1}}, \nabla_{\gamma_1}]W
  = \nabla_{\beta_1}([\nabla_{\beta^{\prime}}, \nabla_{\gamma_{1}}]) 
+ [\nabla_{\beta_1}, \nabla_{\gamma_1}](\nabla)_{\beta^{\prime}}W,
\]
where $\beta^{\prime} = (\beta_2, \ldots, \beta_{l+1})$. The general case of \eqref{eq:commrel}
follows from this case from a separate induction argument on $m$ using
\begin{align*}
  [(\nabla)_{\beta}, \nabla_{\gamma_1}\nabla_{\gamma_2}\cdots \nabla_{\gamma_{m+1}}]W
& = [(\nabla)_{\beta}, \nabla_{\gamma_1}](\nabla)_{\gamma^{\prime}}W
  + \nabla_{\gamma_1}[(\nabla)_{\beta}, (\nabla)_{\gamma^{\prime}}]W
\end{align*}
where $\gamma^{\prime}= (\gamma_2, \ldots, \gamma_{m+1})$.

For \eqref{eq:dtdm}, we combine the general formula for the evolution of the Christoffel symbols with
the equation \eqref{eq:kcf}, to obtain
\[
  \pdt \Gamma_{ij}^k = \frac{1}{2}g^{mk}\left\{\nabla_{i}\pdt g_{jm} + \nabla_{j}\pdt g_{im} - \nabla_{m}\pdt g_{ij}\right\}
  = P^{2k+1}_g(R).
\]
Then we have 
\[
  \left[\pdt, \nabla\right] W = \pdt\Gamma \ast W  = P^{2k+1}_g(R) \ast W,
\]
from which \eqref{eq:dtdm} follows, using the identity
\[
  \left[\pdt, \nabla^{(r+1)}\right] W = \left[\pdt, \nabla\right]\nabla^{(r)}W + \nabla\left[\pdt, \nabla^{(r)}\right]W,
\]
and induction on $r$.
\end{proof}

Now we determine the structure of the evolution equations for the covariant derivatives of the 
curvature tensor. We only need to pay careful attention to the structure of the highest-order terms.
\begin{lemma}\label{lem:rev} Suppose $g(t)$ is a solution to \eqref{eq:kcf}. For any $m = 0, 1, 2, \ldots$,  the $m$-fold covariant derivative $\nabla^{(m)}R$ of $g$
evolves by the equation
\begin{equation}\label{eq:rev}
\pdt \nabla^{(m)}R = (-1)^{k}\left(\Delta^{(k+1)}\nabla^{(m)}R + \alpha(\nabla\nabla \Delta^{(k)}\nabla^{(m)}S)\odot g\right) + P_g^{2k+m +1}(R)
\end{equation}
and the $m$-fold covariant derivative of the scalar curvature satisfies an equation of the form
\begin{equation}\label{eq:sev}
 \pdt \nabla^{(m)}S  = (1 + 2\alpha(n-1))\Delta^{(k+1)}\nabla^{(m)}S + P_{g}^{2k + m + 1}(R)
\end{equation}
on $M\times [0, \Omega]$.
\end{lemma}
Here, for $m=0$, 
\begin{align*}
\begin{split}
  (\nabla\nabla\Delta^{(k)}S\odot g)_{abcd} &= \nabla_a\nabla_d\Delta^{(k)}S g_{bc} 
  + \nabla_b\nabla_c\Delta^{(k)}S g_{ad} - \nabla_a\nabla_c\Delta^{(k)}S g_{bd} \\
  &\phantom{=}- \nabla_b\nabla_d\Delta^{(k)}S g_{ac}
\end{split}
\end{align*}
is the Kulkarni-Nomizu product of $\nabla\nabla\Delta^{(k)}S$ and $g$, and, for $m\geq 1$,
we define
\begin{align*}
 (\nabla\nabla \Delta^{(k)}\nabla^{(m)}S\odot g)_{ab\gamma cd}
&\dfn \nabla_a\nabla_d \Delta^{(k)}(\nabla)_{\gamma}S g_{bc}
+\nabla_b\nabla_c \Delta^{(k)}(\nabla)_{\gamma}S g_{ad}\\
&\phantom{\dfn}-\nabla_a\nabla_c \Delta^{(k)}(\nabla)_{\gamma}S g_{bd}
-\nabla_b\nabla_d \Delta^{(k)}(\nabla)_{\gamma}S g_{ac}
\end{align*}
for any multi-index $\gamma$ of length $m$. 

\begin{proof} The proof is a standard computation, some details of which we include for completeness. We start with the case $m=0$
and the formula 
\begin{equation}\label{eq:drm}
 DR_g[h]_{ijkl} = -\frac{1}{2}\left(\nabla_i\nabla_lh_{jk} + \nabla_j\nabla_kh_{il} - \nabla_i\nabla_kh_{jl} - \nabla_j\nabla_lh_{ik}\right) +  (R\ast_g h)_{ijkl}.
\end{equation}
for the linearization $DR_g$ of $R = \Rm$ at the metric $g$ in the direction of a given symmetric $(2, 0)$-tensor $h$. The Bianchi identities imply that
\[
 DR_g(2\Rc) = -\Delta R + P_g^0(R),   
\]
so, using \eqref{eq:commrel}, we have
\begin{align*}
  DR_g(2\Delta^{(k)}\Rc) &= - \Delta^{(k+1)}R  + P^{2k}_g(R) + \left[DR_g, \Delta^{(k)}\right]\Rc \\
			&= -\Delta^{(k+1)}R + P^{2k}_g(R).
\end{align*}
Similarly, using \eqref{eq:drm}, we see that the second term in $\Theta_{2k}(g)$ contributes terms of the form
\[
  DR_g(2\Delta^{(k)}S g) = -\nabla\nabla\Delta^{(k)}S\odot g + P^{2k}_g(R)
\]
to the evolution of $R$.  

The third term in $\Theta_{2k}(g)$, however, contributes only lower order terms.  To see this, note that 
the diffeomorphism invariance of $R$ implies that $DR_g[\mathcal{L}_Xg] = \mathcal{L}_XR_g$ for any vector field $X$. Since
\[
\Delta^{(k-1)}\nabla\nabla S = \nabla\nabla \Delta^{(k-1)}S + P^{2k-2}_g(R)
\]
by \eqref{eq:commrel}, it follows that
\[
 DR_g[2\Delta^{(k-1)}\nabla\nabla S] = DR_g[\mathcal{L}_{\nabla \Delta^{(k-1)}S} g + P^{2k-2}_g(R)] = P^{2k}_g(R).
\]

Together with the observation that $DR_g[\Lambda_{2k-1}(g)] = P^{2k+1}_g(R)$, we arrive at \eqref{eq:rev} in the 
case $m= 0$.  The identities for $m > 0$ follow from this case and the commutation
relations \eqref{eq:commrel} and \eqref{eq:dtdm}.
\end{proof}

We will also need to express the the difference of $m$-fold covariant derivatives relative to $g$ and $\gt$
in terms of the tensor $A$ and its derivatives.
\begin{lemma}
  Suppose $\nabla$ and $\nabt$ are two connections on $TM$ and $W$ is a smooth tensor field. Define $A = \nabla - \nabt$. Then, for any $m \geq 0$,
\begin{equation}\label{eq:conndiff}
 \nabla^{(m)}W - \nabt^{(m)}W = \sum_{p=1}^{m} Q^{p}(A)\ast \nabt^{(m-p)}W
\end{equation}
where $Q^p(A)$ represents some linear combination of natural contractions 
  of tensors of the form $\nabla^{(m_1)}A\otimes\nabla^{(m_2)}A \otimes \cdots \otimes \nabla^{(m_k)}A$
where $m_1 + m_2 + \ldots + m_k + k = p$.
\end{lemma}
\begin{proof} The proof is by induction.  For $m =1$, we have $\nabla W - \nabt W = A\ast W$. Supposing the claim to be true
for some $m \geq 1$, we can verify it for $m+1$ by observing that
\begin{align*}
&  \nabla^{(m+1)}W - \nabt^{(m+1)}W 
      = \nabla(\nabla^{(m)}W - \nabt^{(m)}W) + (\nabla-\nabt)\nabt^{(m)}W\\
      &\quad= \nabla\left(\sum_{p=1}^{m} Q^{p}(A)\ast \nabt^{(m-p)}W\right) + A\ast\nabt^{(m)}W\\
\begin{split}
      &\quad= \sum_{p=1}^{m}\left(\nabla Q^{p}(A)\ast \nabt^{(m-p)}W + Q^{p}(A)\ast A\ast \nabt^{(m-p)}W + Q^{p}(A)\ast \nabt^{(m-p+1)}W\right) \\
      &\quad\phantom{=}+ Q^1(A)\ast\nabt^{(m)}W
\end{split}\\
      &\quad= \sum_{p=1}^{m}\left(Q^{p+1}(A)\ast \nabt^{(m-p)}W + Q^{p}(A)\ast \nabt^{(m-p+1)}W\right) + Q^1(A)\ast\nabt^{(m+1 -1)}W\\
      &\quad= \sum_{p=1}^{m+1} Q^{p}(A)\ast \nabt^{(m+1-p)}W, 
\end{align*}
which is of the desired form.
\end{proof}

\subsection{Evolution equations for $h$, $\nabla^{(l)}A$, and $X^{(l)}$}
The equations satisfied by the elements of our system are rather involved, but, as before, we will only need to keep a careful
account of the highest order terms. For our purposes, it will be sufficient to verify that each of the lower order terms are polynomial in form
and contain one factor
of some element of our system. Since $M$ is compact and our solutions are smooth,
the remaining factors will be uniformly bounded. 

To simplify our presentation, we will use $C$ to denote any smooth uniformly bounded tensor whose internal structure is immaterial to our argument. 
This tensor will
vary in rank and composition from line-to-line (and even term-to-term) in our expressions.  

\begin{lemma} Suppose that $g$ and $\gt$ are solutions to \eqref{eq:kcf} on $M\times [0, \Omega]$.  Then, for any $l\geq 0$, 
$h$ and $\nabla^{(l)}A$ satisfy the evolution equations
\begin{align}
\label{eq:hev}
  \pdt h &= C \ast h + \sum_{p=0}^{2k}C\ast X^{(p)}, \\
\label{eq:aev}
  \pdt \nabla^{(l)}A &= C \ast h + \sum_{p=0}^{l}C\ast\nabla^{(p)}A + 
\sum_{p=0}^{l+1}\sum_{q=0}^{2k}C\ast \nabla^{(p)}X^{(q)},
\end{align}
and 
$X^{(l)}$ and $Z^{(l)}$ satisfy
\begin{align}
\begin{split}
\label{eq:xev}
  \pdt X^{(l)} &=  (-1)^{k}\left(\Delta^{(k+1)}X^{(l)} + \alpha(\nabla\nabla \Delta^{(k)}Z^{(l)})\odot g\right)\\
		&\phantom{=}   + C\ast h + \sum_{p=0}^{m(k, l)}C\ast \nabla^{(p)}A 
+ \sum_{p=0}^{l+1}\sum_{q=0}^{2k}C\ast \nabla^{(p)}X^{(q)},
\end{split}\\
\begin{split}
\label{eq:zev}
  \pdt Z^{(l)} &= (-1)^k(1 + 2\alpha(n-1))\Delta^{(k+1)}Z^{(l)} + C\ast h \\
  &\phantom{=} + \sum_{p=0}^{m(k, l)}C\ast \nabla^{(p)}A + \sum_{p=0}^{l+1}\sum_{q=0}^{2k}C\ast\nabla^{(p)}X^{(q)},   
\end{split}
\end{align}
where $m(k, l) \dfn \max\{2k+1, l\}$.

\end{lemma}
\begin{proof} Recall that $R$ denotes the $(4, 0)$-curvature tensor, so, schematically, both 
$\Delta^{(l)}\Rc$ and $\nabla\nabla \Delta^{(l-1)}S$ are of the form $(g^{-1})^{*(l+1)}\ast \nabla^{(2l)}R$,
and
$\Delta^{(l)}S g$ is of the form $(g^{-1})^{*(l+2)}\ast\nabla^{(2l)}R \ast g$.
So, using that $g^{ij} - \gt^{ij} = -g^{ik}\gt^{jl}h_{kl}$, we have 
\begin{align*}
 \Theta_{2k}(g) - \Theta_{2k}(\gt) &=(g^{-1})^{*(k+1)}\ast \nabla^{(2k)}R - (\gt^{-1})^{*(k+1)}\ast \nabt^{(2k)}\Rt\\
&\phantom{=} + (g^{-1})^{*(k+2)}\ast\nabla^{(2k)}R \ast g -  (\gt^{-1})^{*(k+2)}\ast\nabt^{(2k)}\Rt \ast \gt\\
&= C\ast h  + C\ast X^{(2k)}.
\end{align*}
Since $\Lambda_{2k-1}(g)$ is of a polynomial structure in $g, g^{-1}, R, \ldots, \nabla^{(2k-1)}R$, we have
\[
  \Lambda_{2k-1}(g) - \Lambda_{2k-1}(\gt) = C\ast h + C\ast X^{(0)} + C\ast X^{(1)}
+ \cdots + C\ast X^{(2k-1)},
\]
and \eqref{eq:hev} follows.

Next, since $\pdt \Gamma = g^{-1}\ast\nabla \pdt g$, we compute that
\begin{align*}
  \pdt A &= g^{-1}\ast\nabla \pdt g - \gt^{-1}\ast\nabt \pdt \gt = C\ast h + C\ast A + C\ast\nabla\pdt h\\
	       &= C\ast h + C\ast A + \sum_{p=0}^{2k}\nabla(C\ast X^{(p)}), 
\end{align*}
where we have used \eqref{eq:hev}, together with the identity $\nabla_k h_{ij} = A_{ki}^p\gt_{pj} + A_{kj}^p\gt_{ip}$.
With the commutator identity \eqref{eq:dtdm}, we then see that
\begin{align*}
  \pdt \nabla^{(l)}A &= \left[\pdt, \nabla^{(l)}\right]A + \nabla^{(l)}\pdt A\\
      &= \sum_{p=0}^{l-1} C\ast \nabla^{(p)}A + \nabla^{(l)}\left(C\ast h + C\ast A + \sum_{q=0}^{2k}\nabla(C\ast X^{(q)}) \right)\\
      &= C\ast h + \sum_{p=0}^{l}C\ast\nabla^{(p)}A + \sum_{p=0}^{l+1}\sum_{q=0}^{2k}C\ast\nabla^{(p)} X^{(q)}, 
\end{align*}
which is \eqref{eq:aev}.

For \eqref{eq:xev}, we consider the difference of the highest-order terms in \eqref{eq:rev} first.
By \eqref{eq:conndiff}, $\nabla^{(l)}W - \nabt^{(l)}W = \sum_{p=0}^{l-1}C\ast \nabla^{(p)}A$ for any tensor $W$ (where $C$ depends on $W$), so
\begin{align*}
 &\Delta^{(k+1)}\nabla^{(l)}R - \Deltat^{(k+1)}\nabt^{(l)}\Rt
  = C\ast h + \sum_{p=0}^{2k+1}C\ast\nabla^{(p)}A + \Delta^{(k+1)}X^{(l)},
\end{align*}
and, similarly,
\begin{align*}
\begin{split}
 \nabla\nabla \Delta^{(k)} \nabla^{(l)} S\odot g  -\nabt\nabt \Deltat^{(k)} \nabt^{(l)} \St\odot \gt
 &= \nabla\nabla \Delta^{(k)} Z^{(l)}\odot g + C\ast h \\
 &\phantom{=}+ \sum_{p=0}^{2k+1}C\ast\nabla^{(p)}A.  
\end{split}
\end{align*}
For the difference of the lower-order terms in \eqref{eq:rev}, we note that
\begin{align*}
&  P_g^{2k+l+1}(R) - P_{\gt}^{2k+l+1}(\Rt)= C\ast h + \sum_{q=0}^{2k+l+1}C\ast X^{(q)}\\
&\quad= C\ast h + \sum_{q=0}^{2k}C\ast X^{(q)} + \sum_{q=1}^{l+1}(C\ast\nabla^{(q)} X^{(2k)} + C\ast\nabla^{(q-1)}A),
\end{align*}
where we have used \eqref{eq:conndiff} to obtain the second line.
Combined with our earlier expressions for the higher order terms, we obtain \eqref{eq:xev}.
Equation \eqref{eq:zev} follows similarly, using \eqref{eq:sev}. 
\end{proof}

\subsection{$L^2$-inequalities.}

For the proof of Theorem \ref{thm:energyder}, we will need the following simple versions of G\aa{}rding's inequality and the Gagliardo-Nirenberg inequality.

\begin{lemma} For any integers $a$, $b$, $l\geq 0$ there exists $C_1 > 0$
depending on these parameters, $n$, and the the covariant derivatives of $R$ up to order $l-1$, such that
  \begin{align}\label{eq:gaarding}
\left|(-1)^{l}\left(\Delta^{(l)} W, W\right) - \|\nabla^{(l)}W\|^2\right| \leq C_1\sum_{p=0}^{l-1}\|\nabla^{(p)}W\|^2.
  \end{align}
for any section $W \in C^{\infty}(T^a_b(M))$.
Moreover, for any $0 \leq l < k$, and any $\epsilon > 0$, there exists a constant $C_2 = C_2(\epsilon, k, l, n)$ such that
\begin{equation}\label{eq:gn}
  \|\nabla^{(l)} W\|^2 \leq C_2\| W\|^2 + \epsilon \|\nabla^{(k)}W\|^2.
\end{equation}
\end{lemma}
\begin{proof}
  The proof of \eqref{eq:gaarding} is essentially standard. For $l=1$, one has $-(\Delta W, W) = \|\nabla W\|^2$.
Proceeding by induction, and assuming, that, for some $l > 1$, the inequality to holds for $l-1$,
one may integrate by parts and use \eqref{eq:commrel} to
obtain
\begin{align*}
  &(-1)^{l}\left(\Delta^{(l)}W, W\right) 
%&= (-1)^{l-1}(\nabla_i\Delta^{(l-1)} W, \nabla_i W)\\
  = (-1)^{l-1}(\Delta^{(l-1)}\nabla_i W, \nabla_i W) + (-1)^{l-1}([\nabla_i, \Delta^{(l-1)}]W, \nabla_i W) \\
  &\quad\quad=  (-1)^{l-1}(\Delta^{(l-1)}\nabla W, \nabla W) + \sum_{p + q = 2l-3}(C\ast\nabla^{(p)}R\ast\nabla^{(q)}W,\nabla W) 
\end{align*}
Since the total sum of the orders of the derivatives in each term on the right is $2l-2$,
one can integrate by parts to achieve that no derivative of order greater than $l-1$ appears
on any one factor of $R$ or $W$. Using the induction hypothesis on the first term establishes the inequality for $l$.

Inequality \eqref{eq:gn} is also standard. Let $\epsilon > 0$ be given. The case $k =1$, $l=0$ is trivial.  
Proceeding by induction on $k$, and assuming the inequality to hold for all $0\leq l < k-1$ where $k > 1$, we let $0 \leq l < k$.
Integrating by parts, we obtain 
\[
  \|\nabla^{(k-1)}W\|^2 \leq a\|\nabla^{(k-2)}W\|\|\nabla^{(k)}W\|
\]
for some $a = a(n)$ and so, using Cauchy-Schwarz and the induction hypothesis,
\begin{align*}
  \|\nabla^{(k-1)}W\|^2 &\leq \frac{a^2}{2\epsilon}\|\nabla^{(k-2)}W\|^2 + \frac{\epsilon}{2}\|\nabla^{(k)}W\|^2\\
			&\leq C\|W\|^2 + \frac{1}{2}\|\nabla^{(k-1)}W\|^2 + \frac{\epsilon}{2}\|\nabla^{(k)}W\|^2
\end{align*}
for some $C = C(\epsilon, k, n)$.
Hence
\begin{equation}\label{eq:km1}
  \|\nabla^{(k-1)}W\|^2 \leq  C\|W\|^2 + \epsilon \|\nabla^{(k)}W\|^2.
\end{equation}
This handles the case $l= k-1$.  If $0 \leq l < k-1$, we can use the induction hypothesis again and \eqref{eq:km1} to obtain
\[
  \|\nabla^{(l)}W\|^2 \leq C\|W\|^2 + \|\nabla^{(k-1)}W\|^2 \leq C\|W\|^2 + \epsilon \|\nabla^{(k)}W\|^2.
\]   
\end{proof}

In the next two lemmas, we use the inequality \eqref{eq:gn} to obtain interpolation inequalities for the elements of our system.
\begin{lemma}\label{lem:ainterpolation}
Let $h$ and  $A$ be as defined above. Then, for $0 \leq l \leq k$, there exists
a constant $C$ depending on $k$, $l$, $g$ and $\gt$ such that
\begin{align}\label{eq:ainterpolation}
 \|\nabla^{(l)}A\|^2 \leq C\|h\|^2 + C\|\nabla^{(k)}A\|^2.
\end{align}

\end{lemma}
\begin{proof}
Inequality \eqref{eq:ainterpolation} will follow immediately from \eqref{eq:gn} once we prove 
an inequality of the form  $\|A\|^2 \leq C\|h\|^2 + C\|\nabla^{(k)}A\|^2$.  For this, we use
the identities
\begin{equation}\label{eq:ahrel}
 \nabla h = -\nabla \gt = A\ast \gt, \quad A = \gt^{-1}\ast \nabla h 
\end{equation}
Differentiating the first of these identities $k$ times and using \eqref{eq:gn} yields a constant $C^{\prime}$
such that
\[
   \|\nabla^{(k+1)}h\|^2 \leq C^{\prime}(\|A\|^2 + \|\nabla^{(k)}A\|^2).
\]
Using this together with \eqref{eq:gn} again, we obtain a  constant $C^{\prime\prime}$ such that
\[
 \|\nabla^{(2)}h\|^2 \leq C^{\prime\prime}\left(\|h\|^2 + \|A\|^2 + \|\nabla^{(k)}A\|^2\right).
\]
Hence, using the identities in \eqref{eq:ahrel} and integrating by parts, we obtain
\begin{align*}
  \|A\|^2 &\leq C^{\prime\prime\prime}(\|h\|\|\nabla^{(2)}h\| +\|h\|\|A\|)\\
  &\leq (C^{\prime\prime\prime})^2\left(\frac{C^{\prime\prime}}{2} + 1\right)\|h\|^2 
      + \frac{1}{2C^{\prime\prime}}\|\nabla^{(2)}h\|^2 + \frac{1}{4}\|A\|^2\\
  &\leq C\|h\|^2 + \frac{3}{4}\|A\|^2 + \frac{1}{2}\|\nabla^{(k)}A\|^2
\end{align*}
so $\|A\|^2 \leq C\|h\|^2 + C\|\nabla^{(k)}A\|^2$ as desired.
\end{proof}

\begin{lemma}\label{lem:xzinterpolation}
 Let $X^{(q)}$ and $Z^{(q)}$ be as defined above. For any nonnegative integers $p$, $q$ with $0 \leq p \leq k$ and  
$0 \leq q \leq 2k$, there exists a 
constant $C$ depending on those parameters and the solutions $g$ and $\gt$ such that
\begin{align}
   \label{eq:xinterpolation}
   \|\nabla^{(p)}X^{(q)}\|^2 &\leq C\left(\|h\|^2 + \|\nabla^{(k)}A\|^2 + \|X^{(0)}\|^2 + \|\nabla^{(k)}X^{(2k)}\|^2\right) \\  
   \label{eq:zinterpolation1}
   \|\nabla^{(p)}Z^{(q)}\|^2 &\leq C\left(\|h\|^2 + \|\nabla^{(k)}A\|^2 + \|Z^{(0)}\|^2 + \|\nabla^{(k)}Z^{(2k)}\|^2\right) \\
\label{eq:zinterpolation2}
			     &\leq C\left(\|h\|^2 + \|\nabla^{(k)}A\|^2 + \|X^{(0)}\|^2 + \|\nabla^{(k)}X^{(2k)}\|^2\right)
\end{align}
\end{lemma}

\begin{proof}
For \eqref{eq:xinterpolation}, we first integrate by parts and use \eqref{eq:conndiff} to obtain	
\begin{align*}
  &\|\nabla^{(k)}X^{(0)}\|^2  = (C\ast\nabla^{(2k)}X^{(0)}, X^{(0)})
  =(C\ast X^{(2k)} + C\ast (\nabt^{(2k)}-\nabla^{(2k)})\Rt, X^{(0)})\\
&\quad\quad = \left(C\ast X^{(2k)}, X^{(0)}\right) + \left(\sum_{r=0}^{2k-1}C\ast\nabla^{(r)}A, X^{(0)}\right)\\
&\quad\quad = \left(C\ast X^{(2k)}, X^{(0)}\right) + \left(\sum_{r=0}^{k}C\ast\nabla^{(r)}A, X^{(0)}\right)
      + \left(\nabla^{(k)}A, \sum_{r=1}^{k-1} C\ast \nabla^{(r)}X^{(0)}\right)
\end{align*} 
which, using \eqref{eq:gn}, implies that for all $\epsilon > 0$, there is a $C = C(\epsilon)$
such that
\begin{equation}\label{eq:pureder}
    \|\nabla^{(k)}X^{(0)}\|^2 \leq  C(\|h\|^2 + \|\nabla^{(k)}A\|^2 + \|X^{(0)}\|^2) + \epsilon \|X^{(2k)}\|^2.
\end{equation}
Using $X^{(k)} = \nabla^{(k)}X^{(0)} + (\nabla^{(k)} -\nabt^{(k)})\Rt$ with \eqref{eq:gn}, \eqref{eq:ainterpolation}, and \eqref{eq:pureder}, we obtain
\begin{align}\nonumber
  \|X^{(k)}\|^2 &\leq C(\|\nabla^{(k)}X^{(0)}\|^2 + \|h\|^2 + \|\nabla^{(k)}A\|^2) \\
  \label{eq:xk}
&\leq C(\|h\|^2 + \|\nabla^{(k)}A\|^2 + \|X^{0}\|^2) + \epsilon \|X^{(2k)}\|^2,
\end{align}
and, in the same way, using $X^{(2k)} = \nabla^{(k)}X^{(k)} + (\nabla^{(k)} -\nabt^{(k)})\nabt^{(k)}\Rt$, we obtain
\begin{align*}
  \|X^{(2k)}\|^2  &\leq \|X^{(k)}\|\|\nabla^{(k)}X^{(2k)}\| + C\sum_{p=0}^{k}\|\nabla^{(p)}A\|\|X^{(2k)}\|\\  
  &\leq C(\|h\|^2 + \|\nabla^{(k)}A\|^2 + \|X^{(k)}\|^2 + \|\nabla^{(k)}X^{(2k)}\|^2).
\end{align*}
Thus, choosing $\epsilon$ in \eqref{eq:xk} sufficiently small, we obtain 
\begin{equation}\label{eq:x2k}
 \|X^{(2k)}\|^2 \leq C(\|h\|^2 + \|\nabla^{(k)}A\|^2 + \|X^{(0)}\|^2 + \|\nabla^{(k)}X^{(2k)}\|^2).
\end{equation}

In combination with  \eqref{eq:gn}, inequalities \eqref{eq:pureder}, \eqref{eq:xk}, and \eqref{eq:x2k} establish the estimate \eqref{eq:xinterpolation} 
in the case $p=0$, $q=k$, the case $p=0$, $q=2k$, and the case $0 \leq p \leq k$, $q=0$. We consider the remaining cases in turn.
First, if $p=0$, and $0 \leq q \leq k$, we have
\[
  X^{(q)} = \nabla^{(q)}X^{(0)} + \sum_{r=0}^{q-1}C\ast\nabla^{(r)}A,
\]
from which the desired inequality is obtained as a consequence of \eqref{eq:gn}, \eqref{eq:ainterpolation}, and the cases considered earlier.
Second, if $p=0$, and $k < q \leq 2k$,
then
\begin{equation}\label{eq:bigq}
  X^{(q)} = \nabla^{(q-k)}X^{(k)} + \sum_{r=0}^{q-k - 1}C\ast\nabla^{(r)}A.
\end{equation}
The second term can be controlled by \eqref{eq:ainterpolation},
and, working as above, we obtain 
\begin{align*}
\|\nabla^{(q-k)}X^{(k)}\|^2 &\leq C\left(\|X^{(k)}\|^2 + \|\nabla^{(k)}X^{(k)}\|^2\right)\\
   &\leq C(\left(\|X^{(k)}\|^2 + \|X^{(2k)} + (\nabt^{(k)} -\nabla^{(k)})\nabt^{(k)}\Rt\|^2\right)\\
  &\leq C\left(\|h\|^2 + \|\nabla^{(k)} A\|^2 + \|X^{(k)}\|^2 + \|X^{(2k)}\|^2\right)
\end{align*}
from which we see that the first term on the right of \eqref{eq:bigq} can also be estimated by cases already considered, 
Thus the desired inequality follows for this range of $p$ and $q$ as well.

It remains only to consider the case that $0 < p \leq k$, $0< q\leq 2k$.  When $q = 2k$ and $p = k$,
the desired inequality is trivial, and the case when $p < k$ and $q=2k$ follows from that of $p=0$, $q=2k$ 
using \eqref{eq:gn}. The subcase  $0 < p  \leq k $, $q < 2k$, and $p + q > 2k$ may be reduced to the case  $p < k$ and $q=2k$ by means
of the identity
\[
  \nabla^{(p)}X^{(q)} = \nabla^{(p+q - 2k)}X^{(2k)} + \sum_{r=0}^{p-1}C\ast\nabla^{(r)}A,
\]
and the remaining case, $0 < p \leq k$, $q < 2k$, and $p + q \leq 2k$ may be reduced to the case $p=0$ and $q \leq 2k$ 
by means of the identity
\[
 \nabla^{(p)}X^{(q)} = X^{(p+q)} + \sum_{r=0}^{p- 1}C\ast\nabla^{(r)}A.
\]

The inequality \eqref{eq:zinterpolation1} may be obtained from the above by the same
argument, substituting $Z^{(q)}$ for $X^{(q)}$. Inequality \eqref{eq:zinterpolation2} follows
from \eqref{eq:zinterpolation1} in view of the estimate 
\[
\|\nabla^{(k+1)}Z^{(l)}\| \leq C(\|h\| + \|\nabla^{(k)}A\| + \|\nabla^{(k+1)}X^{(l)}\|),
\]
 which is a simple consequence of the
identity $Z^{(l)} = C\ast h + C\ast X^{(l)}$.

\end{proof}

\subsection{Proof of Theorem \ref{thm:energyder}}
Having collected the necessary evolution equations and inequalities, we are ready to estimate the derivatives of $\Gc$, $\Hc$ and $\Kc$.
We start with $\Gc$. 
\subsubsection{The derivative of $\Gc$.}
Since $M$ is compact, the effect of differentiating the norms $|\cdot| = |\cdot |_{g(t)}$
and measure $d\mu_{g(t)}$ will only be to generate contributions that are bounded above by multiples of the original quantity.
Using \eqref{eq:hev}, \eqref{eq:aev}, and Lemmas \ref{lem:ainterpolation} and \ref{lem:xzinterpolation}
 in conjunction with \eqref{eq:gn},
we have
\begin{align*}
\begin{split}
  \dot{\Gc} &\leq C\Gc + 2\left(\pdt h, h\right) + \left(\pdt\nabla^{(k)}A, \nabla^{(k)}A\right)\\
  &\leq C\Gc + \left(C \ast h + \sum_{p=0}^{2k}C\ast X^{(p)}, h\right) \\
&\phantom{\leq}
+ \left(C \ast h + \sum_{p=0}^{k}C\ast\nabla^{(p)}A + \sum_{p=0}^{k+1}\sum_{q=0}^{2k}C\ast \nabla^{(p)}X^{(q)}, \nabla^{(k)}A\right)\\
  &\leq C(\epsilon)(\Gc + \Hc) + 2\epsilon\|\nabla^{(k+1)}X^{(2k)}\|^2 
\end{split}
\end{align*}
for any $\epsilon > 0$. This is \eqref{eq:gder}.

\subsubsection{The derivative of $\Hc$ when $\alpha = 0$.}
We next compute the derivative of $\Hc$, starting with the simpler case that $\alpha = 0$.
Using \eqref{eq:xev} and \eqref{eq:gaarding}, and integrating by parts, we have that, for any $0\leq l \leq 2k$,
\begin{align*}
\begin{split}
 &\frac{d}{dt}\|X^{(l)}\|^2 \leq C\|X^{(l)}\|^2 + 2(-1)^k(\Delta^{(k+1)}X^{(l)}, X^{(l)}) \\
  &\phantom{\frac{d}{dt}\|X^{(l)}\|^2\leq}
  + \left(C\ast h + \sum_{p=0}^{2k+1}C\ast \nabla^{(p)}A + \sum_{p=0}^{l+1}\sum_{q=0}^{2k}C\ast\nabla^{(p)}X^{(q)}, X^{(l)}\right)
\end{split}\\
\begin{split}
 &\quad\leq  -2\|\nabla^{(k+1)}X^{(l)}\|^2 + C\sum_{p=0}^{k}\|\nabla^{(p)}X^{(l)}\|^2  
+ C\left(\|h\|+ \sum_{p=0}^{k}\|\nabla^{(p)}A\|\right)\|X^{(l)}\| \\
  &\quad\phantom{\leq}+ C\sum_{p=0}^{k+1}\|\nabla^{(k)}A\|\|\nabla^{(p)}X^{(l)}\| 
+ C\sum_{p=0}^{k+1}\sum_{q=0}^{2k}\|\nabla^{(p)}X^{(q)}\|\| X^{(l)}\|\\
&\quad\phantom{\leq} +C\sum_{r=1}^{(l-k)_+}\sum_{q=0}^{2k}\|\nabla^{(k+1)}X^{(q)}\|\|\nabla^{r}X^{(l)}\|.
\end{split}\\
\end{align*}
Substituting $l=0$ and $l=2k$ in the above inequality, and using \eqref{eq:gn} and the interpolation inequalities in Lemmas \ref{lem:ainterpolation} and \ref{lem:xzinterpolation},
we obtain that, for any $\epsilon > 0$, there is a constant $C= C(\epsilon)$ such that
\begin{align*}
 \begin{split}
  \dot{\Hc} \leq C(\Gc + \Hc)  - 2(1-\epsilon)\|\nabla^{(k+1)}X^{(2k)}\|^2.
 \end{split}
\end{align*}
This is \eqref{eq:hder} in the case $\alpha = 0$.  The computation for the formula \eqref{eq:kder}
is entirely analogous, using \eqref{eq:zinterpolation2} to bound $\Kc$ and all positive terms involving $\|\nabla^{(p)}Z^{(q)}\|$
in terms of $\Gc$, $\Hc$, and $\|\nabla^{(k+1)}X^{(2k)}\|$. In fact, the computation for $\dot{\Kc}$
is the same regardless of the value of $\alpha$, since the term of leading order in the evolution of $Z^{(2k)}$ is always only
a multiple of $\Delta^{(k+1)}Z^{(2k)}$.

\subsubsection{The derivative of $\Hc$ when $\alpha \neq 0$.}
For the case that $\alpha\neq 0$, we need an estimate on the second leading order term in \eqref{eq:xev}.
We will prove the following bound.
\begin{proposition}\label{prop:scalarop}
For any $\alpha$ and any $\epsilon > 0$, there is a constant $C$ depending on $\alpha$, $\epsilon$,
and the solutions $g$ and $\gt$ such that
\begin{align}\label{eq:alphaterm2}
\begin{split}
    &2\alpha(-1)^{k}\left(\nabla\nabla \Delta^{(k)}Z^{(l)}\odot g, X^{(l)} \right)\\
&\quad \leq  -4\alpha\|\nabla^{(k+1)}Z^{(l)}\|^2 + \epsilon\|\nabla^{(k+1)}X^{(l)}\|^2 +  C(\Gc + \Hc)
\end{split}
\end{align}
for any $t\in [0, \Omega]$.
\end{proposition}

Using this estimate and carrying over the computations from the case $\alpha =0$ 
for all of the other terms in $\dot{\Hc}$, we obtain \eqref{eq:hder} and \eqref{eq:kder} in the general case,
completing the proof of Theorem \ref{thm:energyder}.
The proof of Proposition \ref{prop:scalarop} is an easy consequence of the following estimate, proven below,
together with inequality \eqref{eq:gaarding}, and the interpolation inequalities \eqref{eq:gn}, \eqref{eq:ainterpolation},
\eqref{eq:xinterpolation}, and \eqref{eq:zinterpolation2}.
\begin{lemma}\label{lem:alphaterm} There is a constant $C$ such that
\begin{align}\label{eq:alphaterm}
\begin{split}
    &\left|\left(\nabla\nabla \Delta^{(k)}Z^{(l)}\odot g, X^{(l)} \right) - 2\left(\Delta^{(k+1)}Z^{(l)}, Z^{(l)}\right)\right| \\
  &\quad\quad \leq C\|\nabla^{(k+1)}Z^{(l)}\|\left(\|h\| + \sum_{p=0}^k\|\nabla^{(p)}A\| + \sum_{p=0}^{k}\sum_{q=1}^{l}\|\nabla^{(p)}X^{(q)}\|\right)
\end{split}
\end{align}
for any $t\in [0, \Omega]$.
\end{lemma}

\begin{proof}[Proof of Lemma \ref{lem:alphaterm}]
  First we note that
\begin{align} 
\begin{split}
 &      g^{ar}g^{bs}g^{ct}g^{du}g^{\gamma\delta}
      \left(\nabla_a\nabla_d \Delta^{(k)}Z^{(l)}_{\gamma}g_{bc}((\nabla)_{\delta}R_{rstu}
 - (\nabt)_{\delta}\tilde{R}_{rstu})\right)\\
&\quad= C\ast h + g^{ar}g^{du}g^{\gamma\delta}\nabla_a\nabla_d \Delta^{(k)}Z^{(l)}_{\gamma}((\nabla)_{\delta}R_{ru} -(\nabt)_{\delta}\tilde{R}_{ru})
\end{split}
\end{align}
where $\gamma$ and $\delta$ represent multi-indices of length $l$ and $g^{\gamma\delta}(\nabla)_{\gamma}V(\nabla)_{\delta}W$ is shorthand for
$g^{\gamma_1\delta_1}\cdots g^{\gamma_l\delta_l}\nabla_{\gamma_1}\cdots\nabla_{\gamma_l}V \otimes \nabla_{\delta_1}\cdots\nabla_{\delta_l}W$.
Hence,
\begin{align}
\begin{split}\label{eq:int0}
&\int_M\left\langle (\nabla\nabla \Delta^{(k)}Z^{(l)})\odot g , X^{(l)} \right\rangle\,d\mu \\
&\;
  = \int_M\left(\langle C \ast h, X^{(l)}\rangle 
+  4 g^{ar}g^{du}g^{\gamma\delta}\nabla_a\nabla_d \Delta^{(k)}Z^{(l)}_{\gamma}(\nabla_{\delta}R_{ru} -\nabt_{\delta}\tilde{R}_{ru})\right)\,d\mu
\end{split}
\end{align}
Now, integrating by parts and using the contracted second Bianchi identity, we simplify 
the second term on the right-hand side as follows:
\begin{align}\nonumber
 & \int_M g^{ar}g^{du}g^{\gamma\delta}\nabla_a\nabla_d \Delta^{(k)}Z^{(l)}_{\gamma}(\nabla_{\delta}R_{ru} -\nabt_{\delta}\tilde{R}_{ru})\,d\mu\\
\begin{split}\nonumber
 &\quad = \int_{M} g^{\gamma\delta}\bigg(\left(\Delta^{(k)}Z^{(l)}_{\gamma}(g^{ar}g^{du}\nabla_d\nabla_a\nabla_{\delta}R_{ru}
	-\gt^{ar}\gt^{du}\nabt_d\nabt_a\nabt_{\delta}\tilde{R}_{ru})\right) \\
&\quad\phantom{=} +\left\langle \Delta^{(k)}Z^{(l)}, C\ast h + C\ast A + C\ast \nabla A \right\rangle\bigg)\,d\mu
\end{split}\\
\begin{split}\label{eq:int1}
&\quad = \int_{M} \bigg(\frac{1}{2}\left\langle \Delta^{(k)}Z^{(l)},  \Delta Z^{(l)}\right\rangle \\
&\quad\phantom{=}
 +g^{\gamma\delta}\left\langle \Delta^{(k)}Z^{(l)}_{\gamma},
	   (g^{ar}g^{du}[\nabla_d\nabla_a, \nabla_{\delta}]R_{ru} 
- \gt^{ar}\gt^{du}[\nabt_d\nabt_a, \nabt_{\delta}]\Rt_{ru})\right\rangle \\	    
&\quad\phantom{=} + \frac{1}{2}g^{\gamma\delta}\left\langle \Delta^{(k)}Z^{(l)}_\gamma, ([\nabla_{\delta}, \Delta]S - 
[\nabt_{\delta}, \Deltat]\St)\right\rangle \\
&\quad\phantom{=}+\left\langle \Delta^{(k)}Z^{(l)}, C\ast h + C\ast A + C\ast \nabla A \right\rangle\bigg)\,d\mu.\\
\end{split}
\end{align}
Using \eqref{eq:commrel}, we see that $[\nabla^{(2)}, \nabla^{(l)}]R = P^{l}_g(R)$, so the commutator terms
 in the second line of \eqref{eq:int1} can be rewritten as
\begin{align*}
 &g^{ar}g^{du}[\nabla_d\nabla_a, \nabla_{\delta}]R_{ru} - \gt^{ar}\gt^{du}[\nabt_d\nabt_a, \nabt_{\delta}]\Rt_{ru}\\
&\qquad\qquad 
= P^{l}_g(R) - P_{\gt}^l(\Rt) = C\ast h + \sum_{p=0}^{l}C\ast X^{(p)}.
\end{align*}
The commutator terms in the third line are of the same schematic form and can be simplified in the same way.  

Thus, after integrating by parts in the first
line of \eqref{eq:int1} to move the Laplacian from the right side of the inner product to the left,
and further integrations-by-parts on the second, third, and fourth lines
 to move $(k-1)$-covariant derivatives from the factor of $\Delta^{(k)}Z$ to the opposite factor in the inner product,
we obtain
\begin{align*}
& \left|\int_M g^{ar}g^{du}g^{\gamma\delta}\nabla_a\nabla_d \Delta^{(k)}Z^{(l)}_{\gamma}(\nabla_{\delta}R_{ru} 
-\nabt_{\delta}\tilde{R}_{ru})\,d\mu
 - \frac{1}{2}\left(\Delta^{(k+1)}Z^{(l)}, Z^{(l)}\right)\right|\\
&\qquad\qquad \leq
C\|\nabla^{(k+1)}Z^{(l)}\|
\left(\|h\| + \sum_{p=0}^{k}\|\nabla^{(p)}A\| + \sum_{p=0}^k\sum_{q=0}^{l}\|\nabla^{(p)}X^{(q)}\|\right).
\end{align*}
Substituting this expression into \eqref{eq:int0}, we obtain \eqref{eq:alphaterm}.
\end{proof}

\section{A general uniqueness theorem}

As we remarked in the introduction, the method of the previous section does not depend on the specific structure
of \eqref{eq:kcf}, but rather of the structure of the inequalities satisfied
by the prolonged system derived from it. The uniqueness assertion for the prolonged system is essentially then
a consequence of the standard energy argument for strictly parabolic equations.
We formulate a somewhat more general version of the argument below which may be useful for other applications.

\subsection{Setup}
Let $M = M^n$ be a closed manifold equipped with a family of smooth metrics $g(t)$ for $t\in [0, \Omega]$,
and $\Xc$ and $\Yc$ be tensor bundles over $M$.
For simplicity of notation, we will regard $\Xc$ and $\Yc$ as orthogonal subbundles
of $\Wc \dfn \Xc\oplus \Yc$.
Denote by $(U, V)$ the family of $L^2(d\mu_{g(t)})$-inner products
\[
      (U, V) \dfn \int_M \langle U, V\rangle_{g(t)}\,d\mu_{g(t)}
\]
induced by $g(t)$ on $\Wc$ for $t\in [0, \Omega]$, with $\|U\|^2 \dfn (U, U)$. Below, we use 
$P = P(T_1, T_2, \ldots, T_{2k})$ and $Q = Q(T_1, T_2, \ldots, T_{k+1})$ to denote polynomial expressions in of their tensorial arguments, 
that is, finite linear combinations of contractions of tensor products of various subsets of their arguments with respect
to $g(t)$.
\begin{theorem}\label{thm:genversion} 
Suppose that $X = X(t)$ and $Y = Y(t)$ are smooth families of sections of $\Xc$ and $\Yc$
defined for $t\in [0, \Omega]$ which satisfy a system of the form
\begin{align*}
%\label{eq:xysys}
\begin{split}
    \pd{X}{t} - \Lc X &= P(X, \nabla X, \ldots, \nabla^{(2k-1)}X, Y)\\
    \pd{Y}{t}  &= Q(X, \nabla X, \ldots, \nabla^{(k)} X, Y),
\end{split}
\end{align*}
where $\Lc = \Lc(t): C^{\infty}(\Xc) \to C^{\infty}(\Xc)$ is a strongly elliptic 
linear operator of order $2k$ for some $k \geq 1$ with smoothly varying coefficients.
Then, $X(0) = 0$ and $Y(0) = 0$ imply $X(t)= 0$ and $Y(t) = 0$ for all $t\in [0, \Omega]$.
\end{theorem}

Above, the family of operators $\Lc(t)$ are assumed to be strongly elliptic for each $t$, relative to a modulus of ellipticity that is independent of time.
We will prove a slightly more general statement, which permits simpler choices of $X$ and $Y$ in some applications (see e.g., \cite{KotschwarRFU}).

\begin{theorem}\label{thm:genineqversion}
 Suppose that $X = X(t)$ and $Y= Y(t)$ are smooth families of sections of $\Xc$ and $\Yc$ defined
for $t\in [0, \Omega]$ which satisfy a system of the form
\begin{align}
\label{eq:xyineqsys}
\begin{split}
    \left\|\pd{X}{t} - \Lc X - F\right\|
      &\leq C\sum_{p=0}^{k}\|\nabla^{(p)}X\| + C\|Y\|,\\
  \left\|\pd{Y}{t}\right\|  &\leq C\sum_{p=0}^{k}\|\nabla^{(p)}X\| + C\|Y\|
\end{split}
\end{align}
on $M\times [0, \Omega]$
for some constant $C$,
where $\Lc = \Lc(t)$ is a strongly elliptic linear operator of order $2k$ for some $k \geq 1$ with smoothly varying coefficients, and $F$ is a family of sections of 
$\Xc$ satisfying
\begin{equation}\label{eq:fineq}
 (F, X) \leq C\left(\sum_{p=0}^{k}\|\nabla^{(p)}X\| + \|Y\|\right)
\left(\sum_{p=0}^{k-1}\|\nabla^{(p)}X\| + \|Y\|\right).
\end{equation}
Then, $X(0) = 0$ and $Y(0) = 0$ implies $X(t)= 0$ and $Y(t) = 0$ for all $t\in [0, \Omega]$.
\end{theorem}
Theorem \ref{thm:genversion} indeed follows from this restatement, taking $F = P(X, \nabla X, \ldots, Y)$,
since (using the compactness of $M$), we can write 
\[
(P, X) = \sum_{p=0}^{2k-1}(C\ast \nabla^{(p)}X + C\ast Y, X)
\]
for some bounded (possibly zero) tensors $C$. Integrating by parts $p-k$ times on each term with $p > k$ verifies \eqref{eq:fineq}.
In the next section, in which $\Lc$ has order $2$, we will take $F$ in the form $F = \operatorname{div} U$ where  $U$ satisfies 
$U\leq C(|X| + |Y|)$.

\begin{proof}[Proof of Theorem \ref{thm:genineqversion}] The proof is a trivial modification of the standard
version for $Y \equiv 0$.
  Form the quantity $\Ec(t) \dfn \|X\|^2 + \|Y\|^2$. Since $M$ is closed, the contributions of the time-derivatives of $g$,
and $d\mu_g$ to the following computation
are all bounded, and therefore, integrating by parts and using the Cauchy-Schwarz inequality, we have
\begin{align*}
\begin{split}
  \dot{\Ec}&\leq C\Ec + 2\left(\pd{X}{t}, X\right) + 2\left(\pd{Y}{t}, Y\right)\\
&\leq C\Ec + 2\left(\pd{X}{t} - \Lc X - F, X\right) + 2(\mathcal{L}X, X)
+  2\left(F, X\right)
 + 2\left(\pd{Y}{t}, Y\right)\\
	   &\leq C\Ec +2(\Lc X, X) + C\left(\sum_{p=0}^k\|\nabla^{(p)}X\| + \|Y\|\right)
  \left(\sum_{p=0}^{k-1}\|\nabla^{(p)}X\| + \|Y\|\right)
\end{split}
\end{align*}
for some constant $C$.
Since $M$ is compact and $\Lc$ is strongly elliptic with a modulus of ellipticity that is uniform in $t$,
we may apply G\aa{}rding's inequality, which guarantees that 
there is some $\epsilon > 0$ such that
\[
(\Lc X , X) \leq -\epsilon \|\nabla^{(k)}X\|^2 + C\|X\|^2
\]
for all $t\in [0, \Omega]$. A  proof of this inequality for for strongly elliptic operators on $\RR^n$ can be found in \cite{Giaquinta};
the statement for operators on vector bundles over compact $M$ follows from a partition of unity argument.  (See the discussion in Section 2.1 of \cite{BahuaudHelliwellU}
and \cite{Taylor}, p.~461.)

Using Cauchy-Schwarz, we thus obtain
\[
  \dot{\Ec} \leq C\Ec  -\epsilon \|\nabla^{(k)}X\|^2 + C_1\sum_{p=1}^{k-1}\|\nabla^{(p)} X\|^2
\]
for some constant $C_1 = C_1(\epsilon)$.
On the other hand, inequality \eqref{eq:gn} implies
that there is a constant $C = C(\epsilon, k, C_1)$ such that,
for all $0 < p < k$,
\[
   \|\nabla^{(p)}X\|^2 \leq \frac{\epsilon}{C_1(k-1)} \|\nabla^{(k)}X\|^2 + C\|X\|^2.
\]
Therefore, we obtain $\dot{\Ec}\leq C\Ec$
on $[0, \Omega]$, and the claim follows.
\end{proof}
 
\section{The cross-curvature flow}
As an application of Theorem \ref{thm:genineqversion},
we give a proof of the uniqueness of solutions to the 
\emph{cross-curvature flow} of strictly positively or negatively curved metrics on a closed three-manifold $M = M^3$,
describing the prolongation procedure in detail. We first need to introduce some notation.

\subsection{The equation}
Suppose that $g$ has either strictly positive or strictly negative sectional curvature and define $\sigma$
to be $1$ if the curvature is positive and $-1$ otherwise. If $E(g) = \Rc(g) - S/2g$ is the Einstein metric of $g$,
then $\sigma E(g)$ will be negative definite.  We will use the notation 
$\operatorname{E}$ for 
the endomorphism $\operatorname{E}: TM\to TM$ given by $\operatorname{E}_i^j = g^{jk}E_{ik}$.

Let $V \in C^{\infty}(T^0_2(M))$ be the inverse of $E^{ij} = g^{ik}g^{kl}E_{ij}$, i.e., the tensor satisfying $V_{ik}E^{kj} = \delta^j_i$, 
and let $P \dfn \det(\operatorname{E})= \det(E_{ij})/\det(g_{ij})$. The \emph{cross-curvature tensor} of $g$
is then defined to be 
\begin{equation}\label{eq:xcdef}
    X_{ij} \dfn P V_{ij} = -\frac{1}{2}E^{pq}R_{pijq}.
\end{equation}
The \emph{cross-curvature flow} of a family of metrics $g(t)$ is the equation
\begin{align}
  \label{eq:xcf}
      \pdt g &= -\sigma 2X(g).
\end{align}
This equation was introduced by Hamilton and Chow in \cite{ChowHamiltonXCF} as a tool
to study three-manifolds of negative curvature. The short-time existence of solutions to the equation beginning
at metrics of positive or negative curvature was verified by Buckland \cite{Buckland}.

Here we give a proof of the following uniqueness assertion.
\begin{theorem}\label{thm:xcfuniqueness}
Let $M = M^3$ be a closed manifold.
 Suppose $g(t)$, $\gt(t)$ are two solutions to \eqref{eq:xcf} on $M\times [0, \Omega]$
with strictly positive or negative sectional curvature and $g(0) = \gt(0) = \bar{g}$. Then $g(t) = \gt(t)$
for all $t\in [0, \Omega]$.
\end{theorem}

The proof in the positively and negatively curved cases are virtually the same, so we consider here only the case that $\bar{g}$
has strictly negative curvature. As we have observed, in this case, the Einstein tensors $E$, $\tilde{E}$ of $g$ and $\gt$ are 
positive definite.  Let $\lambda > 0$ be a constant such that
\begin{equation}\label{eq:ellipticity}
 \lambda g \leq E \leq \lambda^{-1}g, \quad \lambda g \leq \tilde{E} \leq \lambda^{-1}g
\end{equation}
on $M\times [0, T]$. (Since the manifold $M$ is compact, the solutions $g$ and $\gt$ are uniformly equivalent on $M$ for $t\in [0, \Omega]$.)

\subsection{The prolonged system}

To use Theorem \ref{thm:genineqversion}, we must encode the problem
of uniqueness into one for a prolonged system of the form \eqref{eq:xyineqsys}.  
This can be done in several ways. We find it convenient
to base the parabolic part (the $X$ component) of the system on the difference of the inverses $V$, $\Vt$ of the Einstein tensors of $g$ and $\gt$.
Thus, we define
\begin{equation}\label{eq:xcfxysys}
 h \dfn g- \gt, \quad A \dfn \Gam - \Gamt, \quad W\dfn V - \Vt,
\end{equation}
and introduce the operator $\Box \dfn E^{ab}\nabla_a\nabla_b$. (In general, we take
$\Box = -\sigma E^{ab}\nabla_a\nabla_b$.)
These quantities together satisfy a closed system of differential inequalities relative to the norms 
and connection induced by $g(t)$.
\begin{proposition}\label{prop:xcfsys}
  Under the assumptions of Theorem \ref{thm:xcfuniqueness}, the tensors $W$, $h$, $A$ satisfy the inequalities 
\begin{align*}
 \left|\pdt W - \Box W - \operatorname{div} U\right| &\leq  C\left(|h| + |A| + |W| + |\nabla W|\right)\\
 \left|\pdt h\right| &\leq C(|h|+ |W|)\\
  \left|\pdt A\right| &\leq C(|h|+ |A| + |W| + |\nabla W|)
\end{align*}
for some constant $C$, where $U$ is the $(2, 1)$-tensor given by
\[
 U_{ij}^a = (E^{ab} - \Et^{ab})\nabt_a\Vt_{ij} - E^{ab}A_{ai}^p\Vt_{pj} - E^{ab}A_{aj}^p\Vt_{ip}.
\]
The tensor $U$ satisfies $|U|\leq C(|A| + |W|)$.
\end{proposition}
We prove Proposition \ref{prop:xcfsys} in the next section.

\begin{proof}[Proof of Theorem \ref{thm:xcfuniqueness}]
 Proposition \ref{prop:xcfsys} implies that $\hat{X} = W$ and $\hat{Y} = h \oplus A$ satisfy
\begin{align*}
  \left|\pd{\hat{X}}{t} - \Box \hat{X} -\operatorname{div} U\right| &\leq C\left(|\hat{X}| + |\nabla\hat{X}| + |\hat{Y}|\right),\quad
  \left|\pd{\hat{Y}}{t}\right| \leq C\left(|\hat{X}| + |\nabla \hat{X}| + |\hat{Y}|\right)
\end{align*}
where $|U| \leq C(|\hat{X}| +  |\hat{Y}|)$ for some constant $C$.  The operator $\Box$ is strongly elliptic
under our assumptions, and $F = \operatorname{div}(U)$ satisfies
\[
 (F, \hat{X}) = -(U,\nabla \hat{X}) \leq C\|\nabla\hat{X}\|(\|\hat{X}\| + \|\hat{Y}\|).
\]
The result is now an immediate application of Theorem \ref{thm:genineqversion}.
\end{proof}

In fact, when formulated in terms of $\hat{X}$, $\hat{Y}$ and $U$, the uniqueness assertion of Theorem \ref{thm:xcfuniqueness}
essentially follows
from our earlier general theorem in \cite{KotschwarRFU}, modulo the verification of the inequalities in Proposition \ref{prop:xcfsys}; see Remark 15
in that reference.

\subsection{Evolution equations and proof of Proposition \ref{prop:xcfsys}.}

\begin{lemma} Let $g(t)$ be a solution to \eqref{eq:xcf}.  The Levi-Civita connection and the 
tensor $V$ associated to $g$ evolve according to the 
equations
\begin{align}
\label{eq:xconnev}
  \pdt\Gamma_{ij}^k  &= g^{mk}\left\{\nabla_i(P V_{jm}) + \nabla_j(P V_{im}) -\nabla_{m}(P V_{ij})\right\}\\
\label{eq:xvev}
  \pdt V_{ij} &= \Box V_{ij} + (E^{al}E^{kb} - 2E^{ab}E^{kl})\nabla_kV_{ai}\nabla_l V_{bj} + Pg^{kl}(V_{ik}V_{jl} + V_{ij}V_{kl}).
\end{align}
\end{lemma}
\begin{proof}
 Equation \eqref{eq:xconnev} follows from the standard formula for the evolution of the connection and \eqref{eq:xcdef}.
 Equation \eqref{eq:xvev} follows from some routine calculations from the formula
\[
    \pdt E^{ij} = \Box E^{ij} - \nabla_k E^{jl}\nabla_lE^{ik} - Pg^{ij} - \operatorname{tr}_g(X)E^{ij}
\]
from \cite{ChowHamiltonXCF}, using $V_{ik}E^{kj} =\delta_i^j$. 
\end{proof}

Now, the Einstein tensor is divergence-free, i.e., $\nabla_i E^{jk} = 0$, and so we compute that 
(cf. the computations on pp. 6-7 of \cite{KotschwarRFU}) 
\begin{align*}
&  \Box V_{ij} -\Box \Vt_{ij} = \nabla_{a}\left(E^{ab}\nabla_b V_{ij}\right) - \nabt_a\left(\Et^{ab}\nabt_b\Vt_{ij}\right)\\
\begin{split}
    &\quad= \Box (V_{ij} - \Vt_{ij}) + \nabla_aU^{a}_{ij}  +A_{ap}^a\Et^{pb}\nabt_b\Vt_{ij}
 - A_{ab}^p\Et^{ab}\nabt_p\Vt_{ij} 
       - A_{ai}^p\Et^{ab}\nabt_b\Vt_{pj},
 \end{split}
\end{align*}
where $\Et^{ab} = \gt^{ac}\gt^{bd}\widetilde{E}_{cd}$, and 
\[
 U_{ij}^a = (E^{ab} - \Et^{ab})\nabt_a\Vt_{ij} - E^{ab}A_{bi}^p\Vt_{pj} - E^{ab}A_{bj}^p\Vt_{ip}.
\]
Since $E^{ab} - \Et^{ab} = -E^{ak}\Et^{bl}W_{kl}$,
we have $|U| \leq C(|A| + |W|)$ for some constant $C$. Returning to \eqref{eq:xvev} and putting things together, we obtain 
\begin{equation}
 \label{eq:wev}
    \left|\pdt W - \Box W  - \operatorname{div} U\right| \leq C\left(|h| + |A| + |W| + |\nabla W|\right)
\end{equation}
on $M\times [0, \Omega]$ as desired.

The tensor $h$, meanwhile, evolves according to 
\[
  \pdt h = 2(X - \Xt) = 2P(V-\Vt) + 2(P-\Pt)\Vt.
\]
The first term is just $2PW$, and for the second term, we have $|P - \Pt| \leq C(|h| + |W|)$
for some $C$ since the Einstein tensors of $g$ and $\gt$ are uniformly strictly positive on $M$.
Thus
\begin{equation*}
% \label{eq:xhev}
  \left|\pdt h\right| \leq C\left(|h| + |W|\right).
\end{equation*}

Finally, from \eqref{eq:xconnev}, we have the schematic equation
\begin{align*}
 \pdt\Gamma  &= P g^{-1}\ast \nabla V + g^{-1}\ast\nabla P \ast V.	
\end{align*}
Now,
\begin{align*}
 \nabla V -\nabt \Vt &= \nabla W + A\ast \Vt, \quad \nabla P -\nabt \Pt = P\ast E\ast\nabla V - \Pt\ast\Et\ast\nabt\Vt,
\end{align*}
so expanding the last difference on the right into three terms and estimating them as above, we obtain
\[
 \left|\pdt A \right| \leq C(|h| + |A| + |W| + |\nabla W|)
\]
for some $C$. This completes the proof of Proposition \ref{prop:xcfsys}.

\end{document}